\documentclass[12pt]{article}
\usepackage{color}
\usepackage{amsmath}
\usepackage{bbold}
\usepackage{amsfonts,amssymb}
\usepackage{extarrows}
\usepackage{geometry}
\usepackage{authblk} 
\usepackage{hyperref} 
\usepackage{mathrsfs}
\usepackage[utf8]{inputenc}
\def\0{\emptyset}

\newtheorem{theorem}{Theorem}[section]
\newtheorem{definition}[theorem]{Definition}
\newtheorem{lemma}[theorem]{Lemma}
\newtheorem{claim}[theorem]{Claim}

\newtheorem{conjecture}[theorem]{Conjecture}
\newtheorem{proposition}[theorem]{Proposition}

\newenvironment{proof}{{\noindent\it Proof.}}{\hfill $\square$\par}

\DeclareMathOperator{\ex}{\mathrm{ex}}

\usepackage{graphicx}

\usepackage{mathtools} 

\usepackage{enumerate}
\usepackage{enumitem}

\usepackage{
	amsmath,			
	amssymb,			
	enumerate,		    
	graphicx,			
	lastpage,			
	multicol,			
	multirow,			
	pifont,			    
}

\usepackage[numbers]{natbib}

\begin{document}


\title{The maximum number of triangles in graphs without the square of a path}

\author[1]{\small\bf Yichen Wang\thanks{Corresponding author: E-mail: wangyich22@mails.tsinghua.edu.cn}}
\author[2]{\small\bf Ervin Gy\H{o}ri\thanks{E-mail:  gyori.ervin@renyi.hu}}

\affil[1]{\small Department of Mathematical Sciences, Tsinghua University, Beijing 100084, P.R. China.}
\affil[2]{\small HUN-REN Alfréd Rényi Institute of Mathematics, 1053 Budapest, Hungary.}


\date{}

\maketitle\baselineskip 16.3pt

\begin{abstract}
    The generalized Tur\'an number for $H$ of $G$, denoted by $\ex(n,H,G)$, is the maximum number of copies of $H$ in an $n$-vertex $G$-free graph.
    When $H$ is an edge, $\ex(n,H,G)$ is the classical Tur\'an number $\ex(n,G)$.
    Let $P_k$ be the path with $k$ vertices. The square of $P_k$, denoted by $P_k^2$, is obtained by joining the pairs of vertices with distance at most two in $P_k$.

    The Tur\'an number of $P_k^2$, $\ex(n, P_k^2)$, was determined by several researchers.
    When $k=3$, $P_3^2$ is the triangle and $\ex(n, P_3^2)$ is well-known from Mantel's theorem~\cite{mantel1907vraagstuk}.
    When $k=4$, $\ex(n, P_4^2)$ was solved by Dirac~\cite{ExtensionsofTurnstheoremongraphs} in a more general context.
    When $k=5,6$, the problem was solved by Xiao, Katona, Xiao and Zamora~\cite{XIAO20221}.
    For general $k \ge 7$, the problem was solved by Yuan~\cite{YUAN2022103548} in a more general context.

    Recently, Mukherjee~\cite{MUKHERJEE2024113866} determined the generalized Tur\'an number $\ex(n, K_3, P_5^2)$.
    In this paper, we determine the exact value of $\ex(n, K_3, P_6^2)$ and characterize all the extremal graphs for $n \ge 11$.
\end{abstract}


{\bf Keywords:} generalized Tur\'an number, square of path.


\section{Introduction}\label{sec: intro}

In this paper, we consider simple undirected graphs.
For a graph $G$, let $V(G)$ and $E(G)$ denote the vertex set and edge set of $G$, respectively. We use $v(G)$ and $e(G)$ to denote the number of vertices and edges of $G$, respectively. 
That is, $v(G) = |V(G)|$ and $e(G) = |E(G)|$.
For a graph $G$ and a vertex subset $S \subseteq V(G)$, let $G[S]$ denote the subgraph of $G$ induced by $S$.
Let $K_s$ denote the complete graph on $s$ vertices.
Let $T(n,r)$ denote the Tur\'an graph, which is the complete $r$-partite graph on $n$ vertices with parts as equal as possible.
We use $t(n,r)$ to denote the number of edges of $T(n,r)$.

Let $P_k$ be the path with $k$ vertices. The \emph{square} of $P_k$, denoted by $P_k^2$, is obtained by joining the pairs of vertices with distance one or two in $P_k$~(see Figure~\ref{fig: Pk square}).
Similarly, we can define the $p$-th power of a path $P_k$ denoted by $P_k^p$ which is the graph obtained from taking a copy of $P_k$ and joining each pair of vertices of $P_k$ with distance at most $p$.

\begin{figure}
    \centering
    \includegraphics[width=0.8\linewidth]{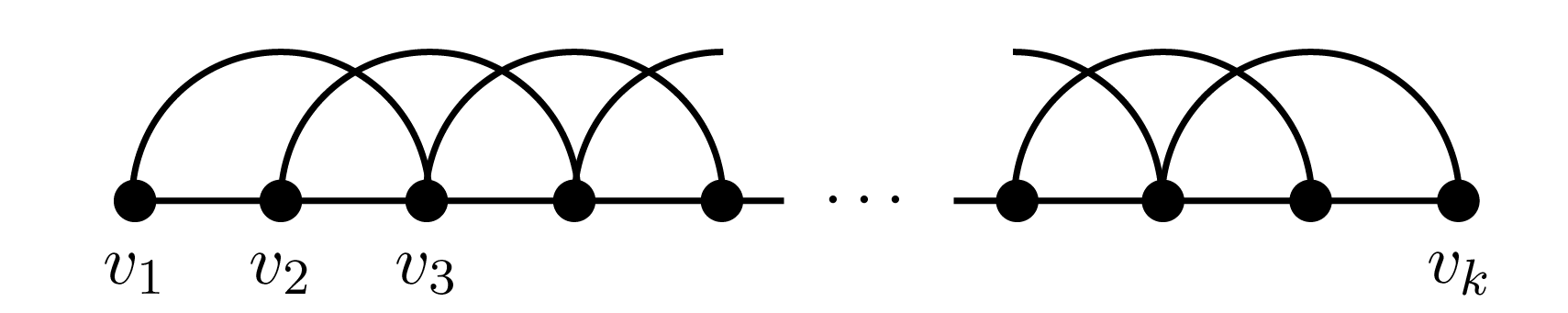}
    \caption{The square of $P_k$ denoted by $P_k^2$.}\label{fig: Pk square}
\end{figure}


Note that $P_3^2$ is the triangle, and the Tur\'an number of triangle is well-known from Mantel's theorem as follows.

\begin{theorem}[\cite{mantel1907vraagstuk}]\label{thm: mantel}
    The maximum number of edges in an $n$-vertex triangle-free graph is $\lfloor n^2/4 \rfloor$, that is $\ex(n, P_3^2) = \lfloor n^2/4 \rfloor = t(n,2)$.
    Furthermore, the only triangle-free graph with $\lfloor n^2/4 \rfloor$ edges is the complete bipartite graph $K_{\lfloor n/2 \rfloor, \lceil n/2 \rceil}$, that is, $T(n,2)$.
\end{theorem}

The case $P_4^2$ was solved by Dirac in a more general context~\cite{ExtensionsofTurnstheoremongraphs}.
The cases $P_5^2$ and $P_6^2$ were solved by Xiao, Katona, Xiao and Zamora~\cite{XIAO20221}. 
Here we state their results for $P_6^2$ as follows.

\begin{theorem}[\cite{XIAO20221}]\label{thm: katona P62}
    Let
    \begin{equation*}
    \begin{aligned}
        f(n) = \left\{ \begin{array}{ll}
            \lfloor \frac{n-1}{2}\rfloor, & n \equiv 1,2,3 \pmod 6, \\
             \lceil \frac{n}{2} \rceil, &\text{otherwise}.\\
        \end{array} \right.
    \end{aligned}
    \end{equation*}
    Then the maximum number of edges in a $P_6^2$-free graph on $n$ vertices~($n \neq 5$) is
    \begin{equation*}
    \begin{aligned}
        \ex(n, P_6^2) = \left\lfloor \frac{n^2}{4} \right\rfloor + f(n) = t(n,2) + f(n).
    \end{aligned}
    \end{equation*}
\end{theorem}

\begin{figure}
    \centering
    \includegraphics[width=\linewidth]{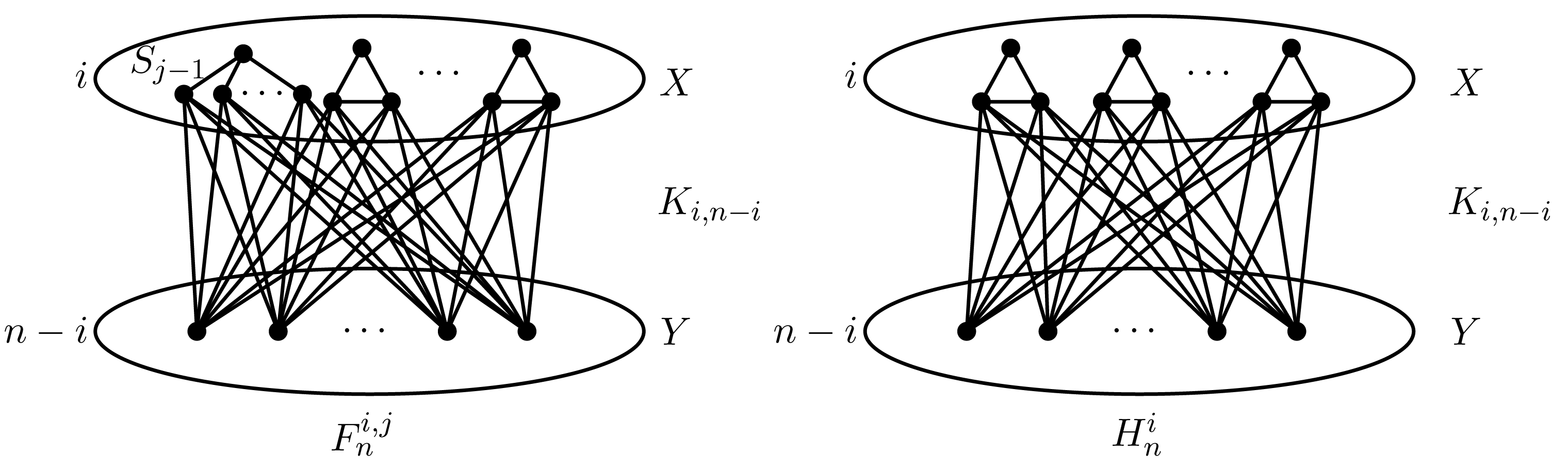}
    \caption{Graphs $F_{n}^{i,j}$ and $H_{n}^{i}$.}\label{fig: Fnij Hni}
\end{figure}

Moreover, they characterized all extremal graphs for $P_6^2$.
\begin{definition}\label{def: Fnij Hni}
    Let $S_k$ denote the star with $k+1$ vertices. 
    Suppose $3 \nmid i$, $3 \mid (i-j)$ and $1 \le j \le i$. Let $F_{n}^{i,j}$ be the graph obtained by adding an $S_{j-1}$ in the class $X$ of size $i$ of $K_{i,n-i}$ and $(i-j)/3$ vertex disjoint triangles in $X$.
    If $3 \mid i$, let $H_{n}^{i}$ be the graph obtained by adding $i$ vertex disjoint triangles in the class $X$ of $K_{i,n-i}$.   
    See Figure~\ref{fig: Fnij Hni} for examples of $F_{n}^{i,j}$ and $H_{n}^{i}$.
\end{definition}

\begin{theorem}[\cite{XIAO20221}]\label{thm: katona P62 extremal structure}
    Let $n \ge 6$ be a natural number.
    Then extremal graphs for $P_6^2$ are the following ones~($j$ can have all values satisfying the conditions $1 \le j \le i$ and $3 \mid (i-j)$).
    \begin{itemize}
        \item When $n \equiv 1 \pmod 6$, then $F_{n}^{\lceil n/2 \rceil, j}$ and $H_{n}^{\lfloor n/2 \rfloor}$.
        \item When $n \equiv 2 \pmod 6$, then $F_{n}^{ n/2 , j}$ and $F_{n}^{n/2 +1, j}$.
        \item When $n \equiv 3 \pmod 6$, then $F_{n}^{\lceil n/2 \rceil, j}$ and $H_{n}^{\lceil n/2 \rceil +1}$.
        \item When $n \equiv 0,4,5 \pmod 6$, then $H_{n}^{n/2}$, $H_{n}^{n/2+1}$ and $H_{n}^{\lceil n/2 \rceil}$, respectively.
    \end{itemize}
\end{theorem}

The more general case $P_k^p$ was solved by Yuan~\cite{YUAN2022103548} when $n$ is sufficiently large as follows.

\begin{theorem}[\cite{YUAN2022103548}]\label{thm: yuan longtu}
Let $n$ be sufficiently large and $s = 2 \lfloor k/(p+1) \rfloor + j$ where $j=1$ if $k \equiv p \pmod p+1$ and $j=0$ otherwise.
Then
\[
    \ex(n, P_k^p) = \max \{ \ex(n_a,P_s) + n_{a}n_{b} + t(n_b,p-1) ~:~ n_a+n_b = n \},
\]
where $t(n_b,p-1)$ is the number of edges of the Tur\'an graph $T_{n_b,p-1}$. Moreover, all extremal graphs are of the form of a complete $p$-partite graph with an extremal graph for $P_s$ added to one of the parts.
\end{theorem}

Note that $P_k^2$ is actually a grid formed by triangles.
Another well-studied triangle grid is the \emph{Triangular Pyramid graph} of $k$ layers denoted by $TP_k$.
\begin{definition}\label{def: pyramid}
    The Triangular Pyramid with $k$ layers, denoted by $TP_k$, is defined as follows: Draw $k+1$ paths in layers such that the first layer is a $1$-vertex path, the second layer is a $2$-vertex path, $\ldots$, and the $(k+1)$-th layer is a $(k+1)$-vertex path. 
    Label the vertices from left to right of the $i$-th layer's path as $x_1^i$, $x_2^i$, $\ldots$, $x_i^i$, where $i \in \{1,2,\ldots,k+1\}$.
    The vertex set of $TP_k$ is the set of all vertices of the $k+1$ paths.
    The edge set consists of all edges in the paths.
    Additionally, for any consecutive layers, say the $(i-1)$-th and $i$-th layers, we add the edges $x_r^{i-1}x_r^{i}$ and $x_r^{i-1}x_{r+1}^{i}$ for each $i \in \{1,2,\ldots,k+1\}$ and $1 \le r \le i-1$.
    See Figure~\ref{fig: TP2 TP3 TP4} for examples of $TP_2$, $TP_3$ and $TP_4$.
\end{definition}

\begin{figure}
    \centering
    \includegraphics[width=0.8\linewidth]{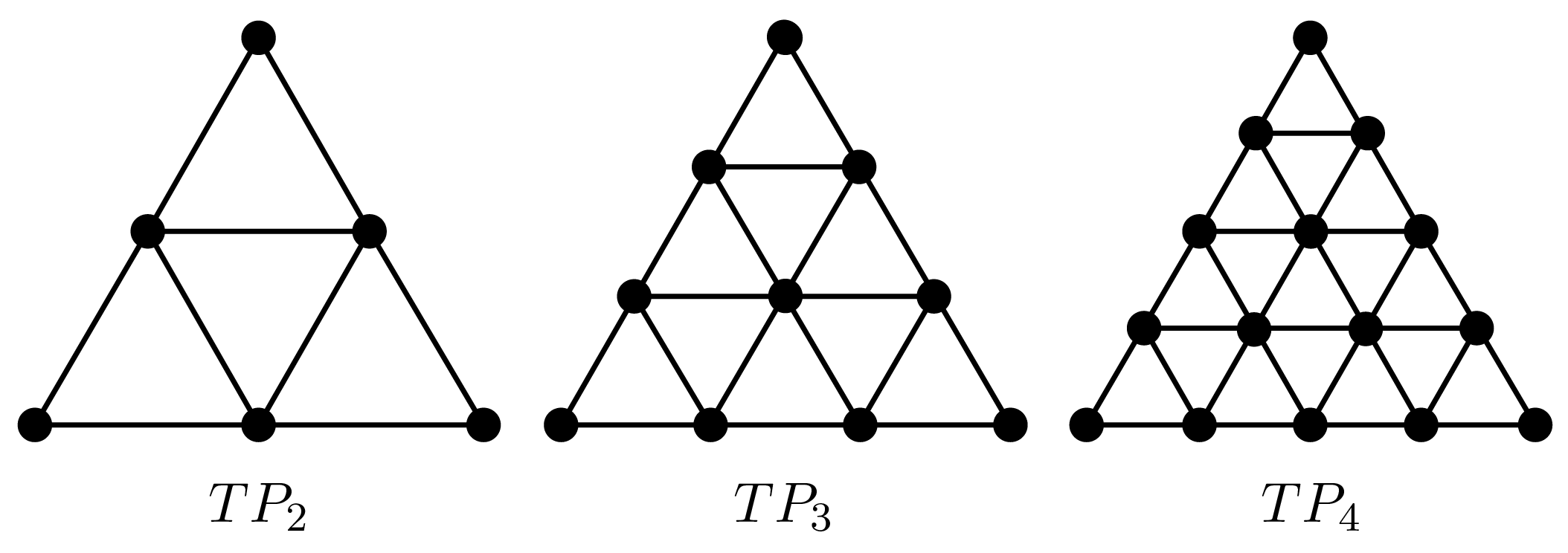}
    \caption{The Triangular Pyramid graphs: $TP_2$, $TP_3$ and $TP_4$.}\label{fig: TP2 TP3 TP4}
\end{figure}

The Tur\'an number of $TP_2$ was determined by Xiao, Katona, Xiao and Zamora~\cite{XIAO20221}, and the Tur\'an number of $TP_3$ was studied by Ghosh, Gy\H{o}ri, Paulos, Xiao and Zamora~\cite{GHOSH202275}.

\begin{theorem}[\cite{XIAO20221}]\label{thm: katona TP2}
    The maximum number of edges in an $n$-vertex $TP_2$-free graph $(n \neq 5)$ is,
    \[
    \ex(n, TP_2)= \left\{ \begin{array}{ll}
        \left \lfloor \frac{n^2}{4} \right \rfloor + \left \lfloor \frac{n}{2} \right \rfloor, & n \not \equiv 2 \pmod 4,\\
        \frac{n^2}{4} + \frac{n}{2} - 1, & n \equiv 2 \pmod 4.\\
    \end{array} \right.
    \]
\end{theorem}

\begin{theorem}[\cite{GHOSH202275}]\label{thm: ghosh TP3}
    The maximum number of edges in an $n$-vertex $TP_3$-free graph is,
    \[
        \ex(n,TP_3) = \frac{1}{4}n^2 + n + o(n).
    \]
\end{theorem}

The generalized Tur\'an number of such triangle grids was also studied recently.
The maximum number of triangles in a $P_5^2$-free graph was determined by Mukherjee~\cite{MUKHERJEE2024113866} as follows.

\begin{theorem}[\cite{MUKHERJEE2024113866}]\label{thm: triangle in P52 Sp4}
    For $n \ge 8$, the maximum number of triangles in an $n$-vertex $P_5^2$-free graph is
    \[
        \ex(n, K_3, P_5^2) = \left\lfloor \frac{n^2}{8} \right\rfloor.
    \]
    For $n=4,5,6,7$, the values of $\ex(n, K_3, P_5^2)$ are $4,4,5,8$, respectively.
\end{theorem}

The maximum number of triangles in a $TP_2$-free graph was studied by Lv, Gy\H{o}ri, He, Salia, Tompkins, Varga and Zhu~\cite{LV2024113682}.

\begin{theorem}[\cite{LV2024113682}]\label{thm: lv edge blow up of triangle}
    Let $n \ge 22$ be an integer. We have
    \[
        \ex(n, K_3, TP_2) \le \frac{n^2}{4}-1 + \mathbb{1}_{4 \mid n},
    \]
    where $\mathbb{1}_{4 \mid n} = 1$ if $4 \mid n$ and $0$ otherwise.
    Furthermore, equality holds when $n$ is even and the unique extremal graph is obtained by taking a complete graph $K_{n/2, n/2}$ and embedding a maximal matching in each part.
\end{theorem}

In this paper, we fully determine the maximum number of triangles in a $P_6^2$-free graph as follows.

\begin{theorem}\label{thm: main P62}
    Let 
    \begin{equation*}
    \begin{aligned}
        g(n) = \left\{ \begin{array}{ll}
            \lfloor \frac{n}{6}\rfloor, & n \equiv 0,1,4 \pmod 6, \\
            \lfloor \frac{n}{6} \rfloor - 1, & n \equiv 2,3 \pmod 6, \\
            \lfloor \frac{n}{6} \rfloor +1, & n \equiv 5 \pmod 6. \\
        \end{array} \right.
    \end{aligned}
    \end{equation*}
    Then the maximum number of triangles in an $n$-vertex~($n \ge 11$) $P_6^2$-free graph is
    \begin{equation*}
    \begin{aligned}
        \ex(n, K_3, P_6^2) = \left \lfloor \frac{n^2}{4} \right \rfloor+g(n) = t(n,2) + g(n).
    \end{aligned}
    \end{equation*}
    And the extremal graphs are the following ones:
    \begin{itemize}
        \item When $n \equiv 0,1 \pmod 6$, then $H_{n}^{ \lfloor n/2 \rfloor}$.
        \item When $n \equiv 2 \pmod 6$, then $H_{n}^{ \lfloor n/2 \rfloor - 1}$.
        \item When $n \equiv 3,4 \pmod 6$, then $H_{n}^{ \lceil n/2 \rceil +1}$.
        \item When $n \equiv 5 \pmod 6$, then $H_{n}^{ \lceil n/2 \rceil}$.
    \end{itemize}
\end{theorem}

\section{Preliminaries}\label{sec: preliminary}

A \emph{block} of $G$ is defined as a maximal collection of triangles of $G$ with the following property: Two triangles $T$, $T'$ are in the same block if they either share an edge or there is a sequence of triangles $T,T_1,\ldots,T_s,T'$ where each triangle~(except the last one) of the sequence shares an edge with the next one.
We will characterize the blocks of $G$.

The idea of blocks was used by Ergemlidze, Győri, Methuku and Salia~\cite{ERGEMLIDZE2017395} studying the maximum number of triangles in $C_5$-free graphs.

The following theorem and propositions are used to characterize the extremal structure in Section~\ref{sec: extremal structure}.

\begin{theorem}[H{\"a}ggkvist, Faudree and Schelp~\cite{haggkvist1981pancyclic}]\label{thm: Haggkvist}
    Let $G$ be a non-bipartite triangle-free graph on $n$ vertices.
    Then \[
    e(G) \le \frac{{(n-1)}^2}{4} + 1.
    \]
\end{theorem}

As a corollary, we have the following proposition.
\begin{proposition}\label{prop: almost bipartite}
    Let $G$ be a triangle-free graph on $n$ vertices~($n \ge 11$).
    When $n \equiv 2,4 \pmod 6$, if $e(G) \ge \frac{n^2}{4} - 1$, then either $G$ is the Tur\'an graph $T(n,2)$, or the graph obtained by deleting one edge from $T(n,2)$, or the graph $K_{n/2-1,n/2+1}$.
    When $n \equiv 3 \pmod 6$, if $e(G) \ge t(n,2) - 2$, then either $G$ is the Tur\'an graph $T(n,2)$, or the graph obtained by deleting one edge from $T(n,2)$, or the graph $K_{\lfloor n/2 \rfloor - 1, \lceil n/2 \rceil + 1}$.
\end{proposition}

\begin{proof}
    When $n \equiv 2,4 \pmod 6$ and $e(G) \ge \frac{n^2}{4} - 1$.
    By Theorem~\ref{thm: Haggkvist}, $G$ is a bipartite graph.
    Let $X$ and $Y$ be the two parts of $G$ with $|X|=x$ and $|Y|=y=n-x$.
    We may assume $x \le y$.
    If $x < n/2-2$, then $e(G) \le xy \le (n/2-2)(n/2+2) = n^2/4 - 4 < t(n,2) - 1$, a contradiction. 
    If $x = n/2 -1$, then $e(G) \le xy = (n/2-1)(n/2+1) = n^2/4 - 1$.
    The equality holds, so $G$ must be $K_{n/2-1,n/2+1}$.
    If $x = n/2$, then either $G$ is the Tur\'an graph $T(n,2)$ or the graph obtained by deleting one edge from $T(n,2)$.

    When $n \equiv 3 \pmod 6$, we can prove it similarly. 
    We omit the details here.
\end{proof}

\begin{proposition}\label{prop: almost bipartite with triangle}
    Let $G$ be a graph obtained by deleting at most two edges from a complete bipartite graph with parts $X$ and $Y$.
    $5 \le |X|, |Y|$.
    If $G$ is a subgraph of $H$ on the same vertex set and $H$ contains a triangle in one part, say $X$.
    Then the maximum number of triangles in $H$ is obtained by embedding a maximal vertex disjoint triangle in $X$ and if there are two remaining vertices, add an edge.
    That is, the maximum is achieved by one of the following:
    \begin{itemize}
        \item When $|X| \equiv 0 \pmod 3$, then $H_{n}^{|X|}$.
        \item When $|X| \equiv i \pmod 3$, then $F_{n}^{|X|,i}$ for $i=1,2$.
    \end{itemize}
\end{proposition}

    \begin{figure}
            \centering
            \includegraphics[width=0.8\linewidth]{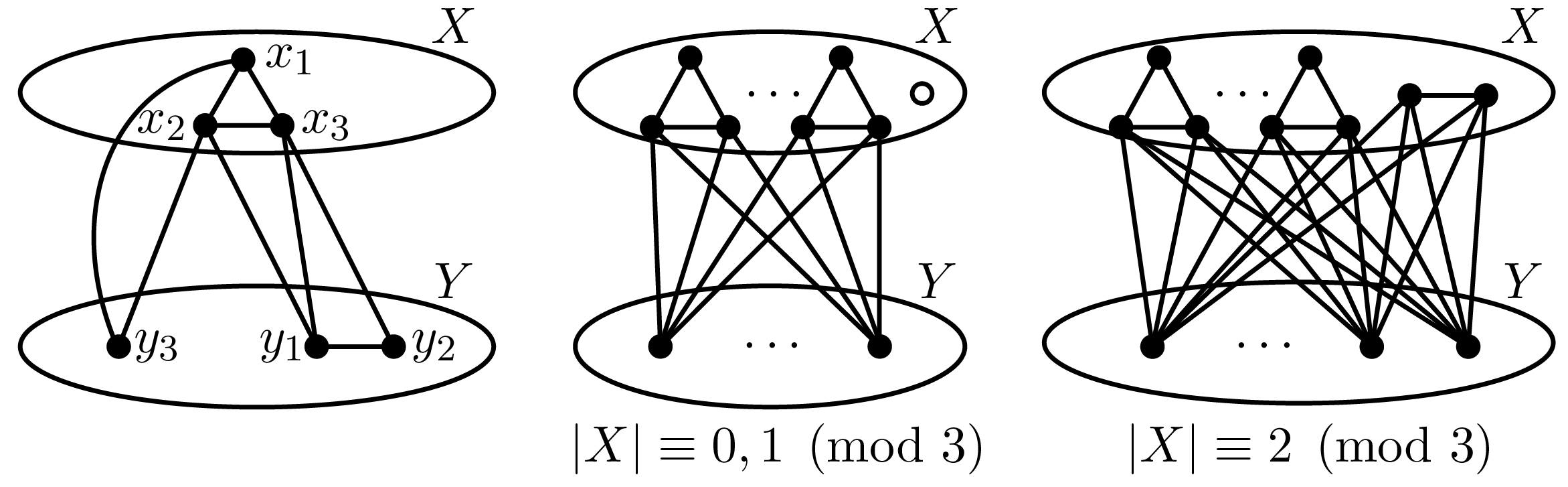}
            \caption{The configurations and extremal structures in Proposition~\ref{prop: almost bipartite with triangle}.}\label{fig: prop}
        \end{figure}

\begin{proof}
    First we claim that $H[X]$ and $H[Y]$ are $P_4$-free.
    Suppose otherwise, let $x_1x_2x_3x_4$ be a $P_4$ in $H[X]$.
    Then there exist two vertices $y_1,y_2 \in Y$ such that $x_{i}y_{j} \in E(H), i=1,2,3,4, j=1,2$ since $|Y| \ge 5$ and there are at most two edges missing between $X$ and $Y$.
    Then $x_1 x_2 y_1 x_3 x_4 y_2$ forms a $P_6^2$ in $H$, a contradiction.

    If $H[X]$ contains a triangle, then there is no edge in $H[Y]$.
    If $H[X]$ contains a triangle $x_1x_2x_3$, we claim that $H[Y]$ is an empty set.
    Suppose otherwise, let $y_1y_2$ be an edge in $H[Y]$~(see Figure~\ref{fig: prop}).
    Since there are at most two edges missing between $X$ and $Y$, there exists a vertex $x \in \{x_1,x_2,x_3\}$ such that $xy_1, xy_2 \in E(H)$. 
    We may assume $x = x_3$.
    Also, we may assume $x_2y_1 \in E(H)$ since $x_1,x_2$ are equivalent and $y_1,y_2$ are equivalent.
    Since $|Y| \ge 4$, there exists a vertex $y_3 \in Y \setminus \{y_1,y_2\}$ such that $x_1y_3,x_2y_3 \in E(H)$.

    Now we know that $Y$ is an independent set in $H$ and $H[X]$ is $P_4$-free and contains a triangle.
    Then $H[X]$ consists of vertex disjoint triangles and vertex disjoint stars.
    A simple calculation shows that the maximum number of triangles in $H$ is obtained by embedding a maximal vertex disjoint triangle in $X$ and if there are two remaining vertices, add an edge.
\end{proof}

Sometimes, the generalized Tur\'an number for triangles is strongly related to the hypergraph Tur\'an number if we take every triangle as a hyperedge.
If we take every triangle in $P_6^2$ as a hyperedge, then we get the $3$-uniform tight path with $4$ hyperedges.
Such problem has been studied by F\"{u}redi, Jiang, Kostochka, Mubayi and Verstra\"{e}te~\cite{Furedi2019}.
Our following discharging method has a similar flavor as theirs.

\section{Proof of Theorem~\ref{thm: main P62}}\label{sec: main proof}

The whole section is devoted to the proof of Theorem~\ref{thm: main P62}.
In this section, we use $t(G)$ to denote the number of triangles in a graph $G$.

\subsection{Relationship between $t(G)$ and $e(G)$}
\begin{lemma}\label{lem: discharging}
    Let $G$ be a $P_6^2$-free graph on $n$ vertices.
    Then $t(G) \le e(G)$.
\end{lemma}

\begin{proof}
    We prove the lemma by discharging method.
First we say an edge is of type $i$~($i=1,2,3$) if there are exactly $i$ triangles containing it. We say an edge is of type $4$ if it is contained in at least $4$ triangles.
The type of a triangle $xyz$ is defined as the triple $(i,j,k)$ where the types of $xy$, $yz$ and $xz$ are $i$, $j$ and $k$.
Then we have the following claim.
\begin{claim}\label{claim: diamond}
        If there are two triangles $xyz$ and $xyw$, where $yz$ is of type $4$, then $xw$ is of type $1$ or $2$.
    \end{claim}

    \noindent
    \textbf{Proof of Claim~\ref{claim: diamond}:}
    Suppose to the contrary that $xw$ is of type $3$ or $4$.
    Then there exists a triangle $xwu$ where $u \notin \{x,y,z,w\}$ since $xw$ is of type $3$ or $4$.
    Moreover, there exists a triangle $yzv$ where $v \notin \{x,y,z,w,u\}$ since $yz$ is of type $4$.
    Finally, $vzyxwu$ forms a $P_6^2$~(see Figure~\ref{fig: diamond}), a contradiction.
    \hfill $\blacksquare$ \par 

    Now we do the discharging process.
    Initially, we assign each edge of $G$ a charge of $1$.
    Then the total charge is $e(G)$.
    Next, each edge distributes its charge to the triangles containing it by the following rules:
    \begin{enumerate}
        \item If $xy$ is of type $i$~($i=1,2,3$), then it gives a charge of $\frac{1}{i}$ to each triangle containing it.
        \item If $xy$ is of type $4$, then consider the triangle $xyz$ containing $xy$.
        \begin{enumerate}[label= (\alph*)]
            \item If $xyz$ is one of the types $\{(1,1,4), (1,2,4),(1,3,4),(1,4,4), (2,2,4)\}$, then $xy$ gives no charge to $xyz$.
            \item If $xyz$ is of type $(2,3,4)$, then $xy$ gives a charge of $\frac{1}{6}$ to $xyz$.
            \item If $xyz$ is of type $(3,3,4)$, then $xy$ gives a charge of $\frac{1}{3}$ to $xyz$.
            \item If $xyz$ is of type $(2,4,4)$, then $xy$ gives a charge of $\frac{1}{4}$ to $xyz$.
            \item If $xyz$ is of type $(3,4,4)$, then $xy$ gives a charge of $\frac{1}{3}$ to $xyz$.
            \item If $xyz$ is of type $(4,4,4)$, then $xy$ gives a charge of $\frac{1}{3}$ to $xyz$.
        \end{enumerate}
    \end{enumerate}

    Clearly, after the discharging process, each triangle receives a charge of at least $1$.
    It remains to prove that after the discharging process, no edge gives out more than a charge of $1$.
    If $xy$ is of type $i$~($i=1,2,3$), then it gives out a total charge of $1$.
    If $xy$ is of type $4$, then consider the following cases.

    \begin{figure}[t]
        \centering
        \begin{minipage}{0.48\linewidth}
            \centering
            \includegraphics[width=0.6\linewidth]{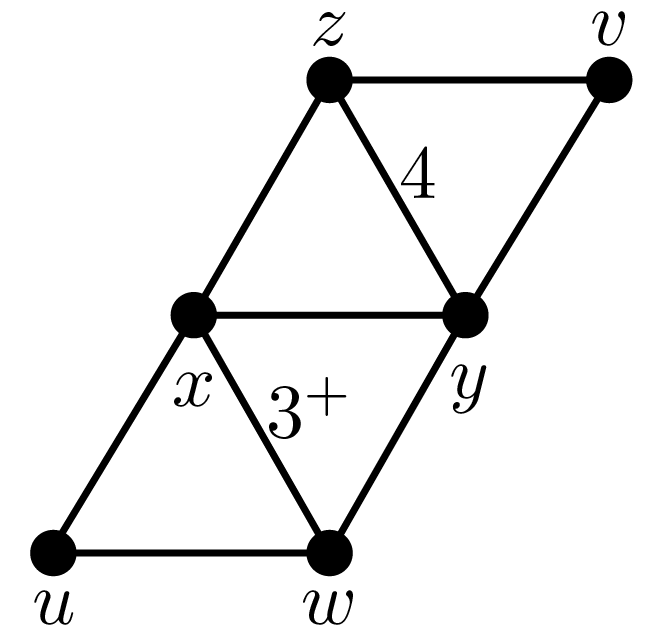}
                \caption{The configuration in the proof of Claim~\ref{claim: diamond}. $3^+$ means the type of $xw$ is at least $3$.}\label{fig: diamond}
        \end{minipage}\hfill
        \begin{minipage}{0.48\linewidth}
            \centering
            \includegraphics[width=0.6\linewidth]{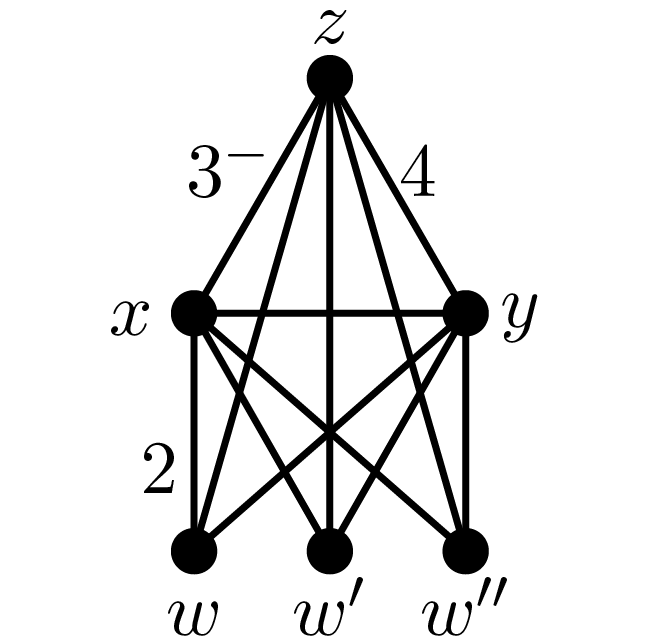}
            \caption{Case 1 in the proof of Lemma~\ref{lem: discharging}. $3^-$ means the type of $xz$ is at most $3$.}\label{fig: case 1 discharging}
        \end{minipage}
    \end{figure}

    \begin{figure}[t]
        \centering
        \begin{minipage}{0.48\linewidth}
            \centering
            \includegraphics[width=\linewidth]{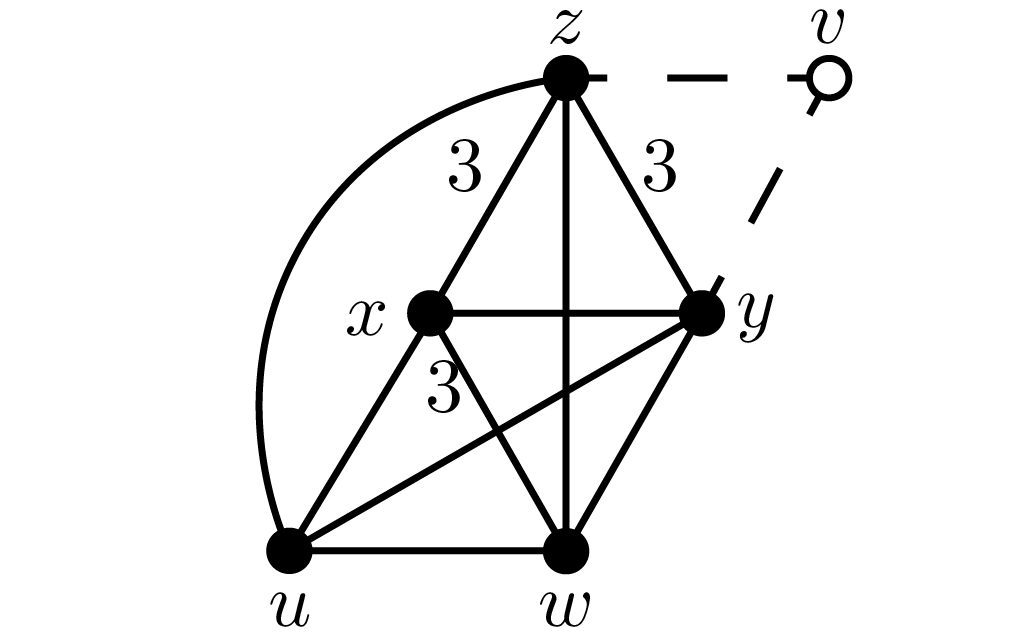}
            \caption{Case 2 in the proof of Lemma~\ref{lem: discharging}.}\label{fig: case 2 discharging}
        \end{minipage}\hfill
        \begin{minipage}{0.48\linewidth}
            \centering
            \includegraphics[width=\linewidth]{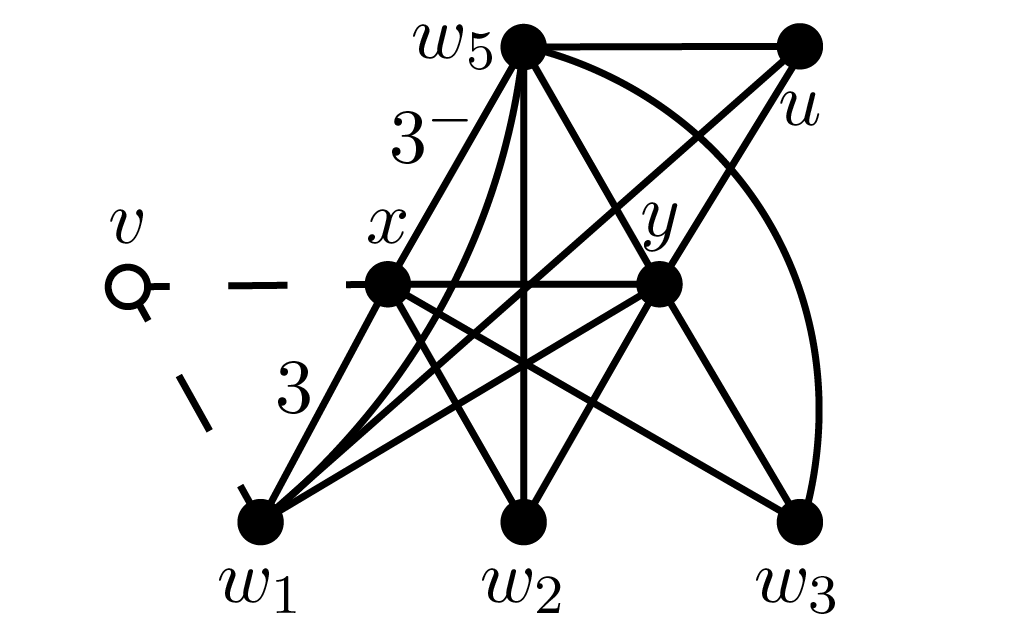}
            \caption{Case 3 in the proof of Lemma~\ref{lem: discharging}. $3^-$ means the type of $xw_5$ is at most $3$.}\label{fig: case 3 discharging}
        \end{minipage}
    \end{figure}

    \textbf{Case 1:} there exists a triangle $xyz$ of type $(i,4,4)$~($i\in \{2,3,4\}$).
    We may assume $yz$ is of type $4$~(see Figure~\ref{fig: case 1 discharging}).
    Then for another triangle $xyw$ which received charge from $xy$, by Claim~\ref{claim: diamond}, $xw$ must be of type $2$~(if $xw$ is of type $1$, then $xyw$ does not receive charge from $xy$).
    
    Moreover, the two triangles containing $xw$ must be $xwy$ and $xwz$~(otherwise, we have a $P_6^2$ as in the proof of Claim~\ref{claim: diamond}).
    Now $xy$ gives at most $\frac{1}{3}$ charge to $xyz$ and at most $\frac{1}{4}$ charge to $xyw$.
    If $xy$ gives more than one charge, then there must exist $w'$ and $w''$ such that $xyw'$ and $xyw''$ are triangles receiving charge from $xy$.
    By a similar argument, $xw'$ and $xw''$ are also of type $2$ and then $zw'$ and $zw''$ are edges.
    Then $xz$ is a type $4$ edge since there are already four triangles containing it: $xyz, xzw, xzw', xzw''$.
    Again by Claim~\ref{claim: diamond}, $yw$ is of type $2$. 
    Then $xyw$ does not receive charge from $xy$, a contradiction.
    In the following, we may assume there is no triangle containing $xy$ of type $(i,4,4)$~($i\in \{2,3,4\}$).

    \textbf{Case 2:} there does not exist a triangle containing $xy$ of type $(i,4,4)$~($i \in \{2,3,4\}$), but exists a triangle $xyz$ of type $(3,3,4)$.
    Assume there is another triangle $xyw$ which received charge from $xy$~(see Figure~\ref{fig: case 2 discharging}). 
    Then $xyw$ is of type $(3,3,4)$ or $(2,3,4)$.
    We may assume $xw$ is of type $3$. Then there exists a vertex $u$ such that $xwu$ is a triangle where $u \notin \{x,y,z,w\}$ since $xw$ is of type $3$.
    Then the three triangles containing $yz$ must be $xyz, yzw, yzu$~(otherwise, we can find a $P_6^2$).
    That is, $xyzuw$ is a $K_5$.
    Then there is no other triangle containing $xy$ that received charge from $xy$, otherwise we can find a $P_6^2$, a contradiction.
    In the following, we may assume there is no triangle containing $xy$ of type $(3,3,4)$.
    
    \textbf{Case 3:} there does not exist a triangle containing $xy$ of type $(3,3,4)$ or $(i,4,4)$~($i \in \{2,3,4\}$).
    Then the triangle received charge from $xy$ can only be of type $(2,3,4)$ and each received $\frac{1}{6}$.
    If $xy$ gives more than one charge, then there exists $w_1,w_2,\ldots,w_7$ such that $xyw_i$ are triangles for each $i$ of type $(2,3,4)$.
    We may assume $xw_1,xw_2,xw_3,xw_4$ are of type $3$~(see Figure~\ref{fig: case 3 discharging}).
    Now we consider $yw_5$.
    Note that $yw_5$ is either of type $2$ or type $3$. There exists a vertex $u \notin \{x,y,w_5\}$ such that $yw_5u$ is a triangle.
    Since $w_1,w_2,w_3,w_4$ are equivalent, we may assume $u \notin \{w_1,w_2,w_3\}$.
    Now we consider $xw_1$, the three triangles containing $xw_1$ must be $xyw_1, xw_1w_5, xw_1u$~(otherwise, we can find a $P_6^2$).
    Then $yw_1w_5$ is a triangle.
    Similarly, we can prove that $yw_2w_5$ and $yw_3w_5$ are triangles.
    Together with the triangle $xyw_5$ containing $yw_5$, we have that $yw_5$ is of type $4$, a contradiction with $xyw_5$ being of type $(2,3,4)$.

    As a conclusion, every edge gives out at most a charge of $1$.
    Thus we have $t(G) \le e(G)$.
\end{proof}
\vspace*{.2cm}

\subsection{Classify the blocks}

Let $G$ be a $P_6^2$-free graph on $n$ vertices with maximum number of edges.
We may assume every edge in $G$ is contained in at least one triangle, otherwise we can delete it without decreasing the number of triangles.
Then every edge is contained in a unique block.
In this subsection, we classify the blocks of a $G$ into four types.
By the lower bound construction in Theorem~\ref{thm: main P62}, we have $t(G) \ge \left\lfloor\frac{n^2}{4}\right\rfloor + g(n)$. 
We first prove the following claim.

\begin{claim}\label{claim: K5-free}
    $G$ is $K_5$-free.
\end{claim}

\begin{proof}
    If there exists a $K_5$ in $G$ with vertices $u_1u_2u_3u_4u_5$, let $U = \{u_1,u_2,u_3,u_4,u_5\}$.
    Then for every vertex in $V(G) \setminus U$, it can be adjacent to at most one vertex in $U$, otherwise we can find a $P_6^2$.

    By Theorem~\ref{thm: katona P62}, we have $e(G-U) \le \ex(n-5, P_6^2) \le \left\lfloor \frac{{(n-5)}^2}{4} \right \rfloor + f(n-5)$.
    Combining that $t(G) \le e(G)$ from Lemma~\ref{lem: discharging}, we have the following inequalities.
    
    \begin{equation*}
    \begin{aligned}
        \left\lfloor\frac{n^2}{4}\right\rfloor + g(n) & \le t(G) \le e(G) \\
            &\le e(G-U) + \binom{5}{2} + (n-5) \\
             &\le \left\lfloor \frac{{(n-5)}^2}{4} \right \rfloor + f(n-5) + n+5,
    \end{aligned}
    \end{equation*}
    a contradiction when $n \ge 11$.
\end{proof}
\vspace*{.2cm}

Let $B$ be a block of $G$.
Suppose we can fix an edge set $E$ in $B$, then we can run the following \emph{growing process} to generate the block $B$:
\begin{itemize}
    \item Start with $B_0 = E$.
    \item While there exists a triangle not contained in $B_i$ but containing an edge $uv$ in $B_i$, then add the triangle to $B_i$ to form $B_{i+1}$. In this case, we say we \emph{grow} a triangle from edge $uv$.
    \item Repeat until no more triangle can be added, then we obtain $B$.
\end{itemize}
If we start the growing process with a non-empty edge set $E$ in $B$, then we can always end up with $B$.
Then we classify the blocks of $G$ with the following claims.

\begin{figure}
        \centering
        \includegraphics[width=0.8\linewidth]{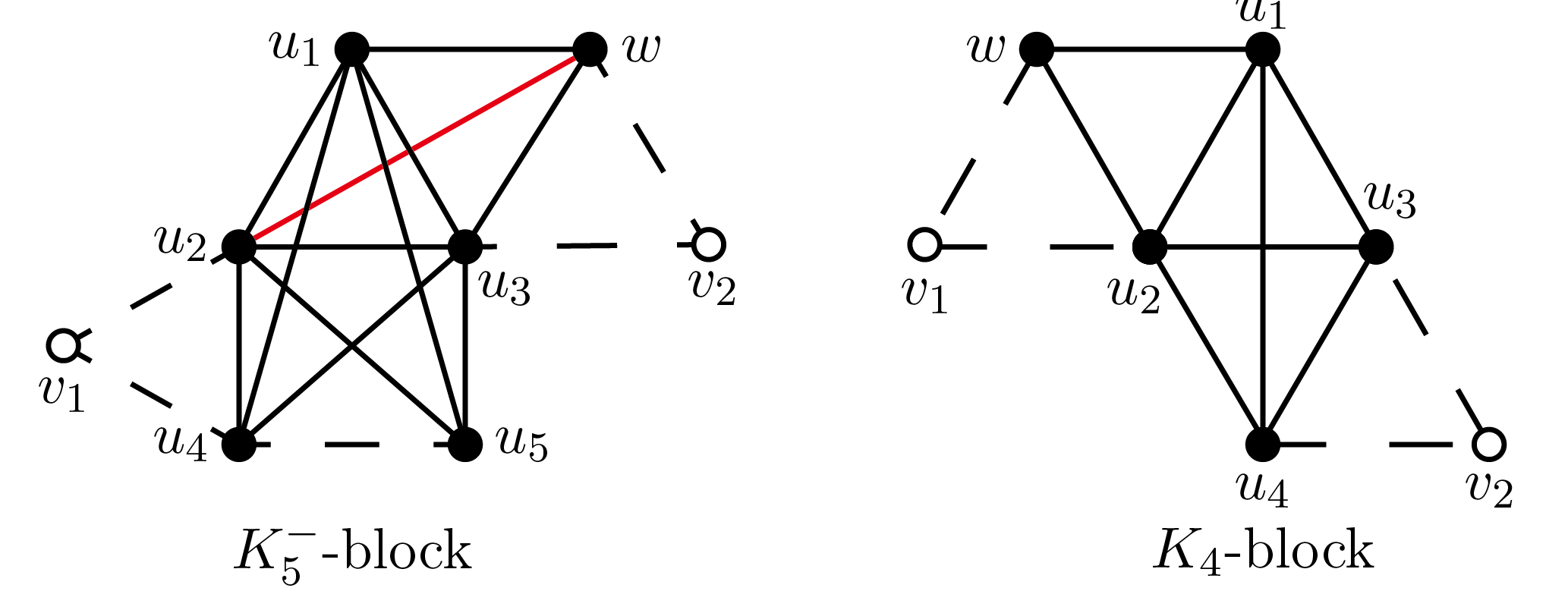}
        \caption{$K_5^-$-block and $K_4$-block.}\label{fig: K5- K4 block}
    \end{figure}

\begin{claim}\label{claim: block K^5^-}
    If a block contains a $K_5^-$~(the graph obtained by deleting one edge from $K_5$), then there exists a triangle $u_1u_2u_3$ in the block such that the remaining vertices in the block are an independent set and every vertex is adjacent to two or three vertices of $u_1u_2u_3$.
    In this case, we call it a \emph{$K_5^-$}-block.
\end{claim}

\begin{proof}
    Suppose we have a $K_5^-$ with vertices $u_1,u_2,u_3,u_4,u_5$ and missing edge $u_4u_5$~(see Figure~\ref{fig: K5- K4 block}).
    Let $U = \{u_1,u_2,u_3,u_4,u_5\}$.
    Then we run the growing process starting with the $K_5^-$ denoted by $B_0$.

    We first claim that $u_2u_4$ cannot grow a triangle during the process~(by symmetry, so do $u_3u_4$, $u_3u_5$ and $u_2u_5$).
    Suppose otherwise there is a triangle $u_2u_4v_1$ where $v_1 \notin \{u_1,u_2,u_3,u_4\}$.
    Since $u_4u_5$ is not an edge, $v_1 \neq u_5$.
    Then $v_1 u_4 u_2 u_3 u_1 u_5$ forms a $P_6^2$, a contradiction.
    
    So we can only grow triangles from $u_1u_2$, $u_1u_3$ or $u_2u_3$.
    Since the three edges are equivalent, suppose $u_1u_3$ grows a triangle $u_1u_3w$ where $w \notin U$.
    Notice that $wu_4$ cannot be an edge, otherwise $u_3u_4w$ can be a triangle grown from $u_3u_4$, a contradiction.
    Similarly, $wu_5$ cannot be an edge.
    $u_2w$ is a possible edge as the red edge shown in Figure~\ref{fig: K5- K4 block}.

    We cannot grow another triangle from $u_3w$~(or equivalently $u_1w$).
    Suppose otherwise, assume $u_3w$ grows a new triangle $u_3wv_2$.
    Then $v_2 w u_3 u_1 u_2 u_4$ forms a $P_6^2$, a contradiction.

    By the above arguments, during the growing process, every vertex in $V(B_i) \setminus U$ is adjacent to two or three vertices of $u_1u_2u_3$ and $V(B_i)\setminus U$ is always an independent set.
    So the claim holds when the growing process ends at $B$.
\end{proof}

\begin{claim}\label{claim: block K^4}
    If a block does not contain $K_5^-$ but contains a $K_4$.
    Let $u_1,u_2,u_3,u_4$ be the vertices of the $K_4$.
    Then there exists three edges forming a triangle~(e.g., $\{u_1u_2,u_2u_3,u_1u_3\}$) or a star~(e.g., $\{u_1u_2,u_1u_3,u_1u_4\}$) in the block such that the remaining vertices in the block are an independent set and every vertex is adjacent to one of the three edges.
    In this case, we call it a \emph{$K_4$}-block.
\end{claim}

\begin{proof}
    Suppose we have a $K_4$ with vertices $u_1,u_2,u_3,u_4$~(see Figure~\ref{fig: K5- K4 block}).
    Let $U = \{u_1,u_2,u_3,u_4\}$.
    Let us run the growing process starting with the $K_4$ denoted by $B_0$.

    Since all the edges in the $K_4$ are equivalent, suppose we grow a new triangle from edge $u_1u_2$, say $u_1u_2w$.
    Notice that $wu_3$ and $wu_4$ cannot be edges, otherwise we have a $K_5^-$.

    Then we claim that we cannot grow triangles from $wu_2$~(or equivalently $wu_1$).
    Suppose otherwise, there is a triangle $wu_2v_1$.
    Since $wu_3$ and $wu_4$ are not edges, $v_1 \notin U$.
    Then $v_1 w u_2u_1u_3u_4$ forms a $P_6^2$, a contradiction.
    
    Moreover, if we already grow a triangle $u_1u_2w$ from $u_1u_2$, then we cannot grow a triangle from $u_3u_4$.
    Suppose otherwise, assume there is a triangle $u_3u_4v_2$ where $v_2 \notin U$.
    Since $wu_3$ is not an edge, $v_2 \notin U \cup \{w\}$.
    Then $v_2 u_4 u_3 u_1 u_2 w$ forms a $P_6^2$, a contradiction.

    Similarly, if we grow a triangle from $u_1u_3$, then we cannot grow a triangle from $u_2u_4$.
    And we cannot grow triangles from both $u_1u_4$ and $u_2u_3$.
    Then there exists three edges, either a triangle or a star, in the $K_4$, such that we can only grow triangles from the three edges~(in the case when we grow triangles from $u_1u_2$ and $u_1u_3$, and one of $\{u_1u_4, u_2u_3\}$).

    Moreover, during the growing process, every vertex in $V(B_i) \setminus U$ is adjacent to exactly two vertices of $u_1u_2u_3$ and $V(B_i)\setminus U$ is always an independent set.
    So the claim holds when the growing process ends at $B$.
\end{proof}

\begin{figure}
    \centering
    \includegraphics[width=0.8\linewidth]{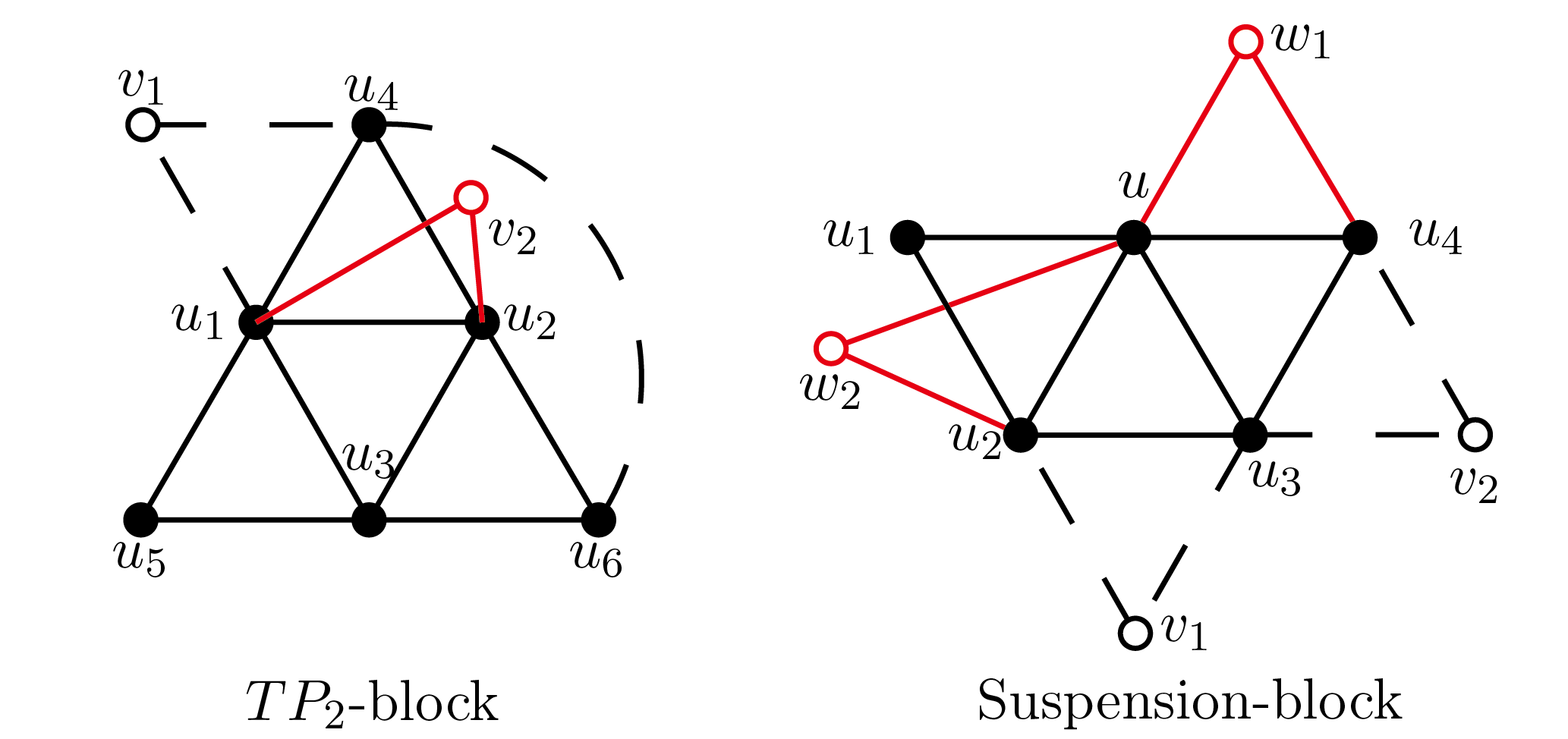}
    \caption{$TP_2$-block and suspension-block.}\label{fig: TP2 and suspension block}
\end{figure}

\begin{claim}\label{claim: block pyramid}
    If a block does not contain $K_5^-$ or $K_4$, but contains a $TP_2$, then there exists a triangle $u_1u_2u_3$ in the block such that the remaining vertices in the block are an independent set and every vertex is adjacent to exactly two vertices of $u_1u_2u_3$.
    In this case, we call it a \emph{$TP_2$}-block.
\end{claim}

\begin{proof}
    Suppose we have a $TP_2$ with vertices $u_1,u_2,u_3,u_4,u_5,u_6$ and $u_1u_2u_3$ is the central triangle~(see Figure~\ref{fig: TP2 and suspension block}).
    Let $U = \{u_1,u_2,u_3,u_4,u_5,u_6\}$.
    We first claim that $u_4u_6$~(or equivalently, $u_4u_5$ and $u_5u_6$) cannot be an edge, otherwise, $u_5 u_1 u_3 u_2 u_6 u_4$ forms a $P_6^2$.
    And $u_4u_3$~(or equivalently, $u_2u_5$ and $u_1u_6$) cannot be an edge, otherwise, we can find a $K_4$.

    Then we run the growing process starting with the $TP_2$ denoted by $B_0$.
    We claim that we cannot grow a triangle from $u_1u_4$~(or equivalently $u_1u_5, u_3u_5, u_3u_6, u_2u_6$ and $u_2u_4$).
    Suppose otherwise, there is a triangle $u_1u_4v_1$.
    By the above arguments, $v_1 \notin U$.
    Then $v_1 u_4 u_1 u_2 u_3 u_6$ forms a $P_6^2$, a contradiction.

    It means, we can only grow triangles from edges in $u_1u_2u_3$.
    Suppose we grow a triangle $u_1u_2v_2$ from $u_1u_2$.
    Also, by the above arguments, $v_2 \notin U$.
    Note that $u_1u_2u_3u_5u_6v_2$ is also a $TP_2$.
    $v_2$ has the same properties as $u_4$.
    And $v_2u_4$ cannot be an edge, otherwise we can find a $K_4$.
    
    As a conclusion, during the growing process, $V(B_i) \setminus U$ is always an independent set and every vertex from $V(B_i) \setminus U$ is adjacent to exactly two vertices of $\{u_1, u_2, u_3\}$.
    So the claim holds when the growing process ends at $B$.
\end{proof}

\begin{claim}\label{claim: block suspension}
    If a block does not contain $K_5^-$, $K_4$ nor $TP_2$, then there exists a vertex $u$ such that every other vertex is adjacent to $u$ and after deleting $u$, the remaining graph is triangle-free.
    In this case, we call it a \emph{suspension}-block.
\end{claim}

\begin{proof}
    It is easy to verify the claim when the block contains at most two triangles.
    So we may assume the block contains at least three triangles.
    We claim that in the block, there must exist a path $u_1u_2u_3u_4$ and a vertex $u$ such that $u$ is adjacent to $u_1,u_2,u_3,u_4$~(see Figure~\ref{fig: TP2 and suspension block}).
    We can start the growing process with a triangle and then it is easy to prove the above statement.

    Now we run the growing process starting with the edge set $E = \{u u_1, u u_2, u u_3, u u_4\}$ denoted by $B_0$.
    During the process, let $N_{B_i}(u)$ be the neighborhood of $u$ in the current block ${B_i}$.
    Note that initially, $V(B_0) = \{u\} \cup N_{B_0}(u)$ and $N_{B_0}(u)$ is a $P_4$.

    We are going to prove that during the process, we maintain the following properties:
    \begin{enumerate}[label= (\arabic*)]
        \item $V(B_i) = \{u\} \cup N_{B_i}(u)$;
        \item $N_{B_i}(u)$ is triangle-free.
        \item every edge in $N_{B_i}(u)$ is contained in a $P_4$ in $N_{B_i}(u)$.
    \end{enumerate}

    The three properties clearly hold for $B_0$.
    Assume the properties hold for $B_i$, we prove they also hold for $B_{i+1}$.
    We first claim that we cannot grow a triangle from an edge not incident to $u$.
    Suppose otherwise, since all the edges not incident to $u$ are contained in $N_{B_i}(u)$, and every edge in $N_{B_i}(u)$ is contained in a $P_4$, we may assume the $P_4$ is $u_1u_2u_3u_4$.
    Then the edge is one of $u_2u_3$, $u_3u_4$ and $u_1u_2$~(by symmetry, we only need to consider $u_2u_3$ and $u_3u_4$).
    If we grow a triangle from $u_2u_3$, say $u_2u_3v_1$, then $v_1 \notin B_i$, otherwise we have a $K_4$.
    Then we have a $TP_2$ with vertex set $\{v_1, u_3, u_2, u, u_1, u_4\}$, a contradiction.
    If we grow a triangle from $u_3u_4$, say $u_3u_4v_2$, then again $v_2 \notin B_i$, and then $v_2 u_4 u_3 u u_2 u_1$ forms a $P_6^2$, a contradiction.
    Thus we can only grow triangles from edges incident to $u$.
    If it brings a new vertex $v$ to the block, then clearly property (1) holds.
    Then property (2) also holds, otherwise we have a $K_4$.

    Finally, if we grow a triangle from edge $u w$. 
    $uw$ is contained in some triangle $u w w'$ in $B_i$.
    And $ww'$ is contained in a $P_4$ in $N_{B_i}(u)$ by property (3).
    We may assume the $P_4$ is $u_1u_2u_3u_4$.
    Then $w \in \{u_1,u_2,u_3,u_4\}$.
    If $w = u_4$~(or equivalently $w = u_1$), assume we grow a triangle $uu_4w_1$.
    Then $w_1 \notin B_i$, otherwise we have a $K_4$.
    Then we bring a new edge $w_1u_4$ to $N_{B_{i+1}}(u)$.
    Clearly, $w_1u_4$ is contained in a $P_4$ within $N_{B_{i+1}}(u)$: $w_1 u_4 u_3 u_2$.
    If $w = u_2$~(or equivalently $w = u_3$), assume we grow a triangle $uu_2w_2$.
    Then $w_2 \notin B_i$, otherwise we have a $K_4$.
    Then we bring a new edge $w_2u_2$ to $N_{B_{i+1}}(u)$.
    Clearly, $w_2u_2$ is contained in a $P_4$ within $N_{B_{i+1}}(u)$: $w_2 u_2 u_3 u_4$.
    Thus property (3) also holds for $B_{i+1}$.

    As a conclusion, the three properties hold for every $B_i$ during the growing process.
    In particular, they hold for the final block $B$, which proves the claim.
\end{proof}

Now we have classified every block of $G$ into four types: $K_5^-$-block, $K_4$-block, $TP_2$-block and suspension-block.

\subsection{Upper bound of $t(G)$ and extremal graphs}\label{sec: extremal structure}

\begin{figure}
        \centering
        \includegraphics[width=\linewidth]{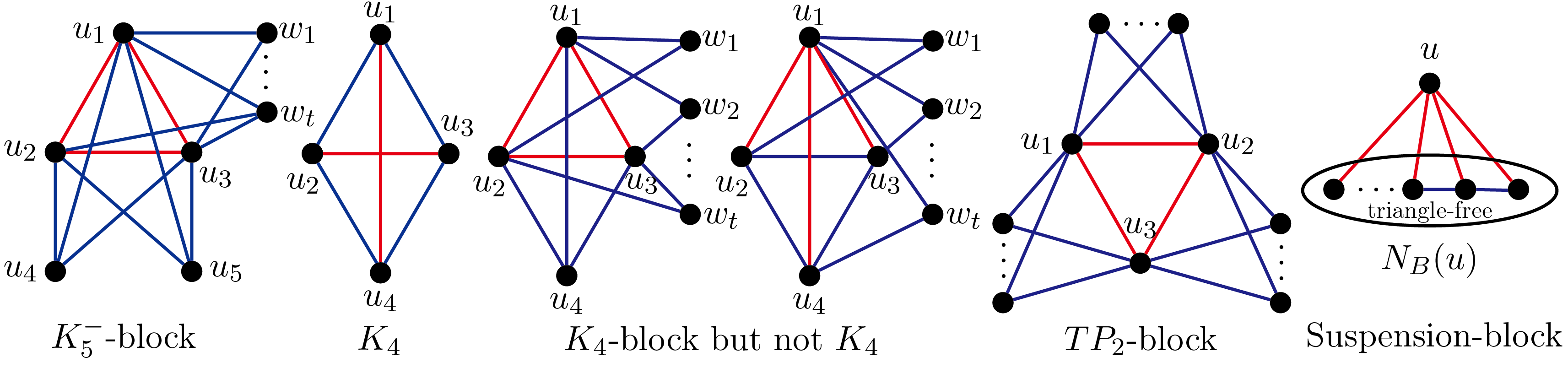}
        \caption{Color the edges from each block with blue and red.}\label{fig: color}
    \end{figure}

Now we color the edges of $G$ as follows~(see Figure~\ref{fig: color}):
\begin{itemize}
    \item for a $K_5^-$ block, color the edges of the triangle $u_1u_2u_3$ red, and color the other edges blue.
    \item for a $K_4$ block which is just a $K_4$ with vertices $u_1u_2u_3u_4$, then color $u_1u_2$ and $u_3u_4$ red, and color the other edges blue.
    \item for a $K_4$ block which is not a $K_4$, color the three edges in Claim~\ref{claim: block K^4} red, and color the other edges blue;
    \item for a pyramid block, color the triangle $u_1u_2u_3$ red, and color the other edges blue;
    \item for a suspension block, color the edges incident to the central vertex red, and color the other edges blue;
\end{itemize}

In the following, let $G_b$ be the subgraph of $G$ formed by the blue edges.
The following claim easily follows from the coloring rules.

\begin{claim}\label{claim: triangle vs blue}
    In each block except $K_5^-$-blocks, the number of triangles is at most the number of blue edges.
    In each $K_5^-$-block, the number of triangles is at most one more than the number of blue edges.
    $G_b$ is triangle-free.
\end{claim}

Note each $K_5^-$-block contains a red triangle, and each red triangle can only belong to one $K_5^-$-block.
Let $e_{b}$ be the number of blue edges and $e_r$ be the number of red edges.
Let $B$ be the number of $K_5^-$-blocks.
Since $G_b$ is triangle-free, then by Mantel's Theorem~\ref{thm: mantel} we have $e_b \le t(n,2)$. 
As a conclusion, we have the following inequalities:

\begin{equation}\label{eq: basic inequalities}
\begin{aligned}
    t(G) & \le e_b + B,\\ 
    e(G) & = e_b + e_r \le t(n,2) + f(n) \\
    e_r & \ge 3B,\\
    e_b & \le t(n,2)\\
\end{aligned}
\end{equation}

It is easy to solve Eq~(\ref{eq: basic inequalities}) and then we have
\begin{equation}\label{eq: result of optimization}
\begin{aligned}
    t(G) \le t(n,2) + \left \lfloor \frac{f(n)}{3} \right \rfloor.
\end{aligned}
\end{equation}
Here $t(G)$ achieves maximum when $e_b = t(n,2)$ and $B = \lfloor \frac{f(n)}{3} \rfloor$.
Note that $t(n,2) + \left \lfloor \frac{f(n)}{3} \right \rfloor$ is very close to $t(n,2) + g(n)$.
We only need some more detailed analysis to prove Theorem~\ref{thm: main P62}.
For the values of $t(n,2)$, $f(n)$ and $g(n)$, edge-extremal graphs and triangle-extremal graphs for different $n$, please refer to Table~\ref{tab: final}.

\begin{table}[t]
    \centering
\begin{tabular}{|c|c|c|c|c|c|}
\hline
$n$    & $t(n,2)=\lfloor n^2/4 \rfloor$ & $f(n)$ & $g(n)$ & edge-extremal graph                   & $K_3$-extremal graph \\ \hline
$6k$   & $9k^2$                  & $3k$   & $k$    & $H_{n}^{3k}$                          & $H_{n}^{3k}$         \\ \hline
$6k+1$ & $9k^2 + 3k$             & $3k$   & $k$    & $H_{n}^{3k}$ and $F_{n}^{3k+1,j}$     & $H_{n}^{3k}$         \\ \hline
$6k+2$ & $9k^2+6k+1$             & $3k$   & $k-1$  & $F_{n}^{3k+1,j}$ and $F_{n}^{3k+2,j}$ & $H_{n}^{3k}$         \\ \hline
$6k+3$ & $9k^2+9k+2$             & $3k+1$ & $k-1$  & $H_{n}^{3k+3}$ and $F_{n}^{3k+2,j}$   & $H_{n}^{3k+1}$       \\ \hline
$6k+4$ & $9k^2+12k+4$            & $3k+2$ & $k$    & $H_{n}^{3k+3}$                        & $H_{n}^{3k+3}$       \\ \hline
$6k+5$ & $9k^2+15k+6$            & $3k+3$ & $k+1$  & $H_{n}^{3k+3}$                        & $H_{n}^{3k+3}$    \\ \hline  
\end{tabular}
\caption{The values of $t(n,2)$, $f(n)$, $g(n)$, edge-extremal graphs and triangle-extremal graphs for different $n$.}\label{tab: final}
\end{table}

\textbf{Case 1:} $n \equiv 0,1,5 \pmod 6$.
In this case, $g(n) = \frac{f(n)}{3}$, then Eq~(\ref{eq: result of optimization}) proves that $\ex(n, K_3, P_6^2) = t(n,2) + g(n)$.
Then we characterize the extremal graphs.

In the inequalities Eq~(\ref{eq: basic inequalities}), when $t(G) = \lfloor \frac{n^2}{4} \rfloor + g(n)$, we must have $e_b = \lfloor \frac{n^2}{4} \rfloor$, $e_r = f(n)$ and $B = \lfloor \frac{f(n)}{3} \rfloor$.
Thus, $e(G) = \lfloor \frac{n^2}{4} \rfloor + f(n)$.
By Theorem~\ref{thm: katona P62 extremal structure}, when $n \equiv 0 \pmod 6$, the unique edge-extremal graph is $H_{n}^{n/2}$, so the unique triangle-extremal graph is also $H_{n}^{n/2}$.
When $n \equiv 1 \pmod 6$, the edge-extremal graphs are $H_{n}^{\lfloor n/2 \rfloor}$ or $F_{n}^{\lceil n/2 \rceil, j}$.
A simple calculation shows that the number of triangles in $F_{n}^{\lceil n/2 \rceil, j}$ is less than $\lfloor \frac{n^2}{4} \rfloor + g(n)$.
So the unique triangle-extremal graph is $H_{n}^{\lfloor n/2 \rfloor}$.
When $n \equiv 5 \pmod 6$, the unique edge-extremal graph is $H_{n}^{\lceil n/2 \rceil}$, so the unique triangle-extremal graph is also $H_{n}^{\lceil n/2 \rceil}$.

\textbf{Case 2:} $n \equiv 4 \pmod 6$.
In this case, $g(n) = \frac{f(n)-2}{3}$, Eq~(\ref{eq: result of optimization}) also proves that $\ex(n, K_3, P_6^2) = t(n,2) + g(n)$.

If $e_b \le t(n,2) - 2$, then combining this inequality and Eq~(\ref{eq: basic inequalities}), we have $t(G) \le t(n,2) + g(n) - 1$, a contradiction.

Then $e_b \ge t(n,2) - 1$ and $B \ge g(n)$.
By Proposition~\ref{prop: almost bipartite}, $G_b$ is a complete bipartite graph or can be obtained by deleting one edge from a complete bipartite graph.
Let $X$ and $Y$ be the two parts of $G_b$. Then $|X|,|Y| \ge n/2-1 \ge 5$.
Since $B \ge g(n) \ge 2$, there exists at least one red triangle in one part.
By Proposition~\ref{prop: almost bipartite with triangle}, $G$ is one of $\{F_{n}^{n/2,j}, H_{n}^{n/2+1}, F_{n}^{n/2-1}\}$.
A simple calculation shows that only $H_{n}^{n/2+1}$ achieves the maximum number of triangles.

\textbf{Case 3:} $n \equiv 2 \pmod 6$.
In this case, $g(n) = \lfloor \frac{f(n)}{3} \rfloor -1 $.
So Eq~(\ref{eq: result of optimization}) only gives $t(G) \le t(n,2) + g(n) + 1$.
We need further analysis to prove $\ex(n, K_3, P_6^2) = t(n,2) + g(n)$.

If $e_b \le t(n,2) - 2$, then combining this inequality and Eq~(\ref{eq: basic inequalities}), we have $t(G) \le t(n,2) + g(n) - 1$, a contradiction.
So $e_b \ge t(n,2) - 1$ and $B \ge g(n)$.
By a similar analysis as Case 2, we can conclude that $G$ is one of $\{F_{n}^{n/2,1}, H_{n}^{n/2-1}, F_{n}^{n/2+1,2}\}$.
And only $H_{n}^{n/2-1}$ achieves the maximum number of triangles.
So $H_{n}^{n/2-1}$ is the unique triangle-extremal graph and $t(G) = \ex(n, K_3, P_6^2) = t(n,2) + g(n)$.

\textbf{Case 4:} $n \equiv 3 \pmod 6$.
Like Case 3, Eq~(\ref{eq: result of optimization}) only gives $t(G) \le t(n,2) + g(n) + 1$ in this case.
If $e_b \le t(n,2) - 3$, then combining this inequality and Eq~(\ref{eq: basic inequalities}), we have $t(G) \le t(n,2) + g(n) - 1$, a contradiction.

So $e_b \ge t(n,2) - 2$.
By Proposition~\ref{prop: almost bipartite}, $G_b$ is a complete bipartite graph or can be obtained by deleting one or two edges from a complete bipartite graph.
Let $X$ and $Y$ be the two parts of $G_b$. Then $|X|,|Y| \ge \lfloor n/2 \rfloor -1 \ge 5$.
If $e_b \ge t(n,2) - 1$, then $B \ge g(n) \ge 1$.
If $e_b = t(n,2) - 2$, then $B \ge g(n) + 2$.
Note that if there exists a red triangle with two vertices in one part and one vertex in the other part, then it has to use two missing edges in $G_b$ between $X$ and $Y$.
Then in both cases, there exists a red triangle in one part.
By Proposition~\ref{prop: almost bipartite with triangle}, $G$ is one of $\{F_{n}^{\lceil n/2 \rceil,2}, F_{n}^{\lfloor n/2 \rfloor}, H_{n}^{\lfloor n/2 \rfloor - 1}, H_{n}^{\lceil n/2 \rceil + 1} \}$.
A simple calculation shows that only $H_{n}^{\lceil n/2 \rceil + 1}$ achieves the maximum number of triangles.
And $t(G) = \ex(n, K_3, P_6^2) = t(n,2) + g(n)$.

As a conclusion, in all cases, we have proved Theorem~\ref{thm: main P62}.
\section{Concluding remarks}\label{sec: conclusion}

In this paper, we determined the maximum number of triangles in a $P_6^2$-free graph on $n$ vertices.
Our result holds for $n \ge 11$.
We believe it is possible to determine the exact value of $\ex(n, K_3, P_6^2)$ for every $n$, with more sophisticated analysis.
However, that would need a lot of detailed investigation.
For the sake of simplicity of the paper, we do not include that here.

This result is a fortunate case since the path is not too long, so that we can obtain a good understanding of the local structures.
For example, this can be seen in the discharging proof and in the classification of the blocks.
For general $P_k^2$ with $k \ge 7$, the local structures would be much more complicated.
New techniques may be needed to solve the problem.

Another direction is to consider the maximum number of triangles in the Triangular Pyramid graph $TP_k$.
When studying $\ex(n, K_3, P_6^2)$, we need the Tur\'an number of $P_6^2$ as a key ingredient.
In the studies of $\ex(n,K_3, TP_2)$~\cite{LV2024113682} and $\ex(n, K_3, P_5^2)$~\cite{MUKHERJEE2024113866}, both rely on the Tur\'an number of $TP_2$ and $P_5^2$, respectively.
Since we already have an understanding of $\ex(n, TP_3)$~(Theorem~\ref{thm: ghosh TP3})~\cite{GHOSH202275}, perhaps it is possible to study $\ex(n, K_3, TP_3)$.
For $TP_k$ with $k \ge 4$, even the Tur\'an number $\ex(n, TP_k)$ is still unknown.
In~\cite{GHOSH202275}, Ghosh, Gy\H{o}ri, Paulos, Xiao and Zamora proposed a conjecture about the extremal structure of $TP_4$.
\begin{conjecture}[\cite{GHOSH202275}]\label{conj: TP4 extremal}
    For $n$ sufficiently large, 
    \[
        \ex(n, TP_4) = \frac{n^2}{4} + \Theta(n^{4/3}).
    \]
\end{conjecture}

\section{Acknowledgments}\label{sec: acknowledgment}
The research of Wang is supported by the China Scholarship Council (No. 202506210200) and the National Natural Science Foundation of China (Grant 12571372).

\bibliography{ref.bib}
\bibliographystyle{wyc4}

\end{document}